\documentclass[12pt]{amsart}
\usepackage{amsmath}
\usepackage{amssymb}
\usepackage{latexsym}

\newcommand{\Zint}{\mathbb {Z}}    
     
\newcommand{\Rea}{\mathbb {R}}      
\newcommand{\R}{\mathbb {R}}      
\newcommand{\Cplx}{\mathbb {C}}     
\newcommand{\halmos}{\rule{5pt}{5pt}}

\numberwithin{equation}{section}

\newtheorem{prop}{\bf Proposition}[section]
\newtheorem{thm}[prop]{\bf Theorem}
\newtheorem{lemma}[prop]{\bf Lemma}

\newtheorem{cor}[prop]{\bf Corollary}

\theoremstyle{definition} 

\newenvironment{remk}{\noindent{\bf Remark}\hskip 5pt}{\hfill{$\Box$}}

\begin{document}
\title[Polynomial solutions of $q$-Heun equation and ultradiscrete limit]
{Polynomial solutions of $q$-Heun equation and ultradiscrete limit}
\author{Kentaro Kojima}
\author{Tsukasa Sato}
\author{Kouichi Takemura}
\address{Department of Mathematics, Faculty of Science and Engineering, Chuo University, 1-13-27 Kasuga, Bunkyo-ku Tokyo 112-8551, Japan}
\email{takemura@math.chuo-u.ac.jp}
\begin{abstract}
We study polynomial-type solutions of the $q$-Heun equation, which is related with quasi-exact solvability.
The condition that the $q$-Heun equation has a non-zero polynomial-type solution is described by the roots of the spectral polynomial, whose variable is the accessory parameter $E$.
We obtain sufficient conditions that the roots of the spectral polynomial are all real and distinct.
We consider the ultradiscrete limit to clarify the roots of the spectral polynomial and the zeros of the polynomial-type solution of the $q$-Heun equation.
\end{abstract}
\maketitle

\section{Introduction}

It is widely known that classical orthogonal polynomials play important roles in mathematics and physics.
Among them, the Legendre polynomial, the Chebyshev polynomial, the Gegenbauer polynomial and the Jacobi polynomial are essentially described by the hypergeometric function $_2 F _1(\alpha ,\beta ;\gamma ;z)$, which satisfies the hypergeometric differential equation 
\begin{align}
&  z(1-z) \frac{d^2y}{dz^2} + \left( \gamma - (\alpha + \beta +1)z \right) \frac{dy}{dz} -\alpha \beta  y=0.
\end{align}
It is a standard form of second order Fuchsian differential equation with three regular singularities $\{ 0,1,\infty \}$.

A $q$-difference analogue of the hypergeometric function is written as
\begin{align}
& _2 \phi _1 (a ,b ;c ;x) = \sum_{n=0}^{\infty} \frac{(a ;q)_n (b ;q)_n }{(q;q) _n (c ;q)_n } x^n, \; (\lambda ,q)_n= \prod_{i=0}^{n-1}(1- \lambda q^i),
\end{align}
and it is called the basis hypergeometric function or $q$-hypergeometric function.
The basis hypergeometric function satisfies the basic (or $q$-difference) hypergeometric equation 
\begin{equation}
(x-q) f(x/q) - ((a+b)x -q-c)f(x)+ (abx-c)f(q x)=0. 
\end{equation}
Note that every coefficient of $ f(x/q)$, $f(x)$ and $f(q x) $ is linear in $x$.

A standard form of second order Fuchsian differential equation with four regular singularities $\{ 0,1,t, \infty \}$ is given by
\begin{align}
\frac{d^2y}{dz^2} +  \left( \frac{\gamma}{z}+\frac{\delta }{z-1}+\frac{\epsilon }{z-t}\right) \frac{dy}{dz} + \frac{\alpha \beta z -B}{z(z - 1)(z - t)} y= 0
\label{eq:Heun}
\end{align}
with the condition $\gamma +\delta +\epsilon = \alpha +\beta +1$, and it is called Heun's differential equation.
The parameter $B$ is called an accessory parameter, which is independent from the local exponents.
As we shall explain later, polynomial solutions of the Heun equation have different features.

A $q$-difference analogue of Heun's differential equation was given by Hahn \cite{Hahn} in 1971 as the form
\begin{equation}
a(x) g(x/q) + b(x) g(x) + c(x) g(qx) =0
\label{eq:axgbxgcxg}
\end{equation}
such that $a(x)$, $b(x)$, $c(x)$ are polynomials such that $\deg_x a(x)= \deg_x c(x)=2 $, $a(0) \neq 0 \neq c(0) $ and $\deg _x b(x) \leq 2$.
Recently the $q$-Heun equation was recovered by two methods \cite{TakR}, one is by degeneration of Ruijsenaars-van Diejen operator \cite{vD0,RuiN}, and the other is by specialization of the linear $q$-difference equation related with $q$-Painlev\'e VI equations \cite{JS}.
We adopt the expression of the $q$-Heun equation as
\begin{align}
& (x-q^{h_1 +1/2} t_1) (x- q^{h_2 +1/2} t_2) g(x/q)  + q^{\alpha _1 +\alpha _2} (x - q^{l_1-1/2}t_1 ) (x - q^{l_2 -1/2} t_2) g(qx) \label{eq:RuijD5}
\\
&  \qquad -\{ (q^{\alpha _1} +q^{\alpha _2} ) x^2 + E x + q^{(h_1 +h_2 + l_1 + l_2 +\alpha _1 +\alpha _2 )/2 } ( q^{\beta /2}+ q^{-\beta/2}) t_1 t_2 \} g(x) =0. \nonumber
\end{align}
Note that the parameter $E$ in Eq.(\ref{eq:RuijD5}) may be regarded as an accessory parameter, which was already pointed out by Hahn.
By the limit $q\to 1$, we recover Heun's differential equation (see \cite{Hahn,TakR}).

In this paper we investigate polynomial-type solutions of the $q$-Heun equation.
Here we recall polynomial solutions of Heun's differential equation (\ref{eq:Heun}) by following \cite{CKLT2,Tak1} (see also \cite{WW,Ron}).
Set
\begin{equation}
y=\sum_{n=0}^{\infty} \tilde{c}_{n}x^{n},\quad (c_{0}=1),
\end{equation}
and substitute it to the differential equation (\ref{eq:Heun}).
Then the coefficients satisfy $t \gamma \tilde{c}_{1}   =B \tilde{c}_{0}$ and 
\begin{align}
 t (n+1)(n+\gamma ) \tilde{c}_{n+1}  =& [n\{(n-1+\gamma )(1+t)+  t \delta  + \epsilon \}+B] \tilde{c}_{n} \label{eq:Hlci}\\
 &  -(n-1+\alpha )(n-1+\beta) \tilde{c}_{n-1}, \; (n=1,2,\dots ). \nonumber
\end{align}
If $t \neq 0,1$ and $\gamma \not \in {\mathbb{Z}}_{\leq 0}$, then $\tilde{c}_{n}$ is a polynomial in $B$ of degree $n$ and we denote it by $\tilde{c}_{n}(B)$.
Moreover we assume that $\alpha=-N$ or $\beta=-N$ for some $N \in{\mathbb{Z}}_{\geq0} $.
Let $B_{0}$ be a solution to the equation $\tilde{c}_{N+1} (B)=0$. Then it follows from (\ref{eq:Hlci}) for $n=N+1$ that $\tilde{c}_{N+2} (B_{0})=0$.
By applying (\ref{eq:Hlci}) for $n=N+2, N+3, \dots$, we have $\tilde{c}_{n} (B_{0})=0$ for $n\geq N+3$.
Hence, if $\tilde{c}_{N+1} (B_{0})=0$, then the differential equation (\ref{eq:Heun}) have a non-zero polynomial solution.
More precisely, we have the following proposition.
\begin{prop} $($\cite{WW,Ron}$)$ \label{prop:Heunpolym} 
Assume that $t \not \in \{ 0 ,1 \}$, $\gamma \not \in {\mathbb{Z}}_{\leq 0} $, $(\alpha +N)(\beta +N)=0$ and $N \in{\mathbb{Z}}_{\geq 0} $.
If $B$ is a solution to the equation $\tilde{c}_{N+1} (B)=0$, then the differential equation (\ref{eq:Heun}) have a non-zero polynomial solution of degree no more than $N$.
\end{prop}
We call $\tilde{c}_{N+1} (B)$ the spectral polynomial.
It is fundamental to study the zeros of the spectral polynomial.
Although the spectral polynomial $\tilde{c}_{N+1} (B)$ may have non-real roots or multiple roots in general, we have a sufficient condition that the spectral polynomial has only distinct real roots.
Namely, if $ (\alpha +N)(\beta +N)=0$, $N \in{\mathbb{Z}}_{\geq 0} $, $\delta $,$\epsilon $ and $\gamma $ are real, $\gamma >0$, $\delta + \epsilon  +\gamma +N >1$ and $t <0$, then the equation $\tilde{c}_{N+1} (B)=0$ has all its roots real and distinct.
Note that it can be proved by applying the argument of Sturm sequence (see \cite{CKLT2} for details).

In this paper, we investigate polynomial-type solutions of the $q$-Heun equation, and apply the argument of Sturm sequence in the $q$-deformed case.
Recall that the $q$-Heun equation is written in Eq.(\ref{eq:RuijD5}).
Set $\lambda _1 = (h_1 +h_2 -l_1-l_2 -\alpha _1-\alpha _2 -\beta +2)/2$.
Under the assumption that  $-\lambda _1 - \alpha _1 (:=N)$ is a non-negative integer and $\beta \not \in \{ 1,2,\dots ,N\}$, we obtain the algebraic equation $c_{N+1} (E)=0 $, which is an analogue of $\tilde{c}_{N+1} (B)=0$, such that the $q$-Heun equation has a solution of the form 
\begin{equation}
f(x)= x^{\lambda _1} \sum _{n=0}^{N } c_n (E_0)  x^n ,
\label{eq:gxxlapol0}
\end{equation}
if $E=E_0$ is a solution of the algebraic equation $c_{N+1} (E)=0 $ (see Proposition \ref{prop:prop0}).
We call $c_{N+1} (E)$ the spectral polynomial of the $q$-Heun equation.
We consider real-root property for the spectral polynomial $c_{N+1} (E)$ in Theorem \ref{thm:SturmqHeun}.

In general it would be impossible to solve the roots of the spectral polynomial $c_{N+1} (E)$ explicitly as well as the spectral polynomial $\tilde{c}_{N+1} (B)$ of the Heun equation.
Then we adopt an idea from the ultradiscrete limit $q \to +0$, as the case of the $q$-Painlev\'e equations \cite{TT,IIT,IT}, in order to understand the solutions to the algebraic equation $c_{N+1} (E)=0 $.
In section \ref{sec:UDL}, we obtain the behaviour of the solutions to $c_{N+1} (E)=0 $ as $q \to +0$ in some cases and it is described as $E_k \sim -cq^{d-k}$ $(k=1,2,\dots ,N+1)$ for some $c \in \R _{>0}$ and $d \in \R $ (see Theorems \ref{thm:specpol1} and \ref{thm:specpol2}).
We can also obtain the behaviour of the zeros of the polynomial-type solution of the $q$-Heun equation with the condition $E=E_k$ (see Theorems \ref{thm:eigenf1} and \ref{thm:eigenf2}).
Note that the solutions to $c_{N+1} (E)=0 $ as $q \to +0$ may not be written as $E_k \sim -cq^{d-k}$ $(k=1,2,\dots ,N+1)$ in some cases (see Remark in section \ref{sec:UDL}).

This paper is organized as follows.
In section \ref{sec:loc}, we consider the polynomial-type solution of the $q$-Heun equation and introduce the spectral polynomial $c_{N+1} (E) $.
In section \ref{sec:realroot}, we show the real root property of the spectral polynomial $c_{N+1} (E)$.
In section \ref{sec:UDL}, we analyze the solutions to $c_{N+1}(E)=0 $ by applying the ultradiscrete limit $q \to +0$.
In section \ref{sec:rem}, we give concluding remarks. 
In the appendix, we introduce theorems on the ultradiscrete limit of an algebraic equation and continuation of the solutions.
Throughout this paper, we assume $0<q<1$.

\section{Polynomial-type solutions of the $q$-Heun equation} \label{sec:loc}

Let $A ^{\langle 4 \rangle} $ be the operator defined by 
\begin{align}
 A^{\langle 4 \rangle} = & x^{-1} (x-q^{h_1 + 1/2} t_1) (x-q^{h_2 +1/2} t_2) T_{q^{-1}} \label{eq:qH} \\
& +  q^{\alpha _1 +\alpha _2} x^{-1} (x - q^{l_1 -1/2} t_1) (x - q^{l_2 -1/2} t_2) T_q \nonumber  \\
& -\{ (q^{\alpha _1} +q^{\alpha _2} ) x  + q^{(h_1 +h_2 + l_1 +l_2 +\alpha _1 +\alpha _2 )/2} ( q^{\beta/2} + q^{-\beta/2} ) t_1 t_2 x^{-1} \} , \nonumber 
\end{align}
where $T_{q^{-1}} g(x) =g(x/q)$ and $T_q g(x)=g(qx) $.
Then the $q$-Heun equation is written as
\begin{align}
& ( A^{\langle 4 \rangle} -E) g(x)=0,
\label{eq:A4E}
\end{align}
where $E$ is a constant.
The action of $A ^{\langle 4 \rangle} $ to $x^{\mu } $ is written as
\begin{align}
& A^{\langle 4 \rangle} x^{\mu } =  d^{\langle 4 \rangle, +} (\mu )  x^{\mu +1 } + d^{\langle 4 \rangle, 0} (\mu ) x^{\mu } + d^{\langle 4 \rangle, -} (\mu ) x^{\mu -1} , \label{eq:A4xmu}
\end{align}
where
\begin{align}
& d^{\langle 4 \rangle, +} (\mu ) = q^{-\mu } (1-  q^{\alpha _1 +\mu } )  (1-  q^{\alpha _2 +\mu } ) ,\label{eq:d4+} \\
& d^{\langle 4 \rangle, 0} (\mu ) = -q^{-\mu } \{q^{1/2} (q^{h_1 } t_1 +q^{h_2 } t_2 ) + ( q^{l_1} t_1 +q^{l_2} t_2 ) q^{\alpha _1 +\alpha _2 + 2 \mu -1/2} \} , \nonumber \\
&  d^{\langle 4 \rangle, -} (\mu ) = t_1 t_2 q^{h_1 + h_2 + 1} q^{-\mu } (1-  q^{\mu -\lambda _1} )  (1-  q^{\mu -\lambda _2 } ) \nonumber 
\end{align}
and 
\begin{equation}
\lambda _1 = (h_1 +h_2 -l_1-l_2 -\alpha _1-\alpha _2 -\beta +2)/2, \; \lambda _2 = (h_1 +h_2 -l_1-l_2 -\alpha _1-\alpha _2 +\beta +2)/2 .
\label{eq:A4la1la2}
\end{equation}
In a special case, it is shown in \cite{TakqH} that the operator $A^{\langle 4 \rangle} $ preserves a finite-dimensional space (see \cite{TakQ} for the Heun equation). 
\begin{prop} $($\cite{TakqH}$)$ \label{prop:TakqH}
Let $\lambda \in \{ \lambda _1 ,\lambda _2\}$, $\alpha \in \{  \alpha _1 , \alpha _2 \}$ and assume that $-\lambda - \alpha (:=N)$ is a non-negative integer. 
Let $V^{\langle 4 \rangle}$ be the space spanned by the monomials $x^{\lambda +k} $ $(k=0,\dots ,N)$, i.e.
\begin{equation}
V^{\langle 4 \rangle} = \{ c_0 x^{\lambda } +c_1 x^{\lambda +1} +\cdots + c_N x^{\lambda +N} | c_0 , c_1 , \cdots , c_N \in \Cplx \} .
\end{equation}
Then the operator $A^{\langle 4 \rangle} $ preserves the space $V^{\langle 4 \rangle} $.
\end{prop}
\begin{proof}
The proposition follows from $d^{\langle 4 \rangle, -} (\lambda  ) =0 $ and $d^{\langle 4 \rangle, +} (\lambda +N ) = 0 $.
See \cite{TakqH} for details.
\end{proof}
Note that the values $\lambda _1 $ and $\lambda _2 $ in Eq.(\ref{eq:A4la1la2}) are exponents of $q$-Heun equation (Eq.(\ref{eq:A4E})) about $x=0$, and the values $ \alpha _1$ and $\alpha _2 $ are exponents about $x=\infty $ (see \cite{TakqH}).

Set
\begin{align}
& g(x)= x^{\lambda } \sum _{n=0}^{N} c_n x^n \; (c_0 \neq 0),
\label{eq:gxx0}
\end{align}
and substitute it into Eq.(\ref{eq:A4E}).
Under the assumption of Proposition \ref{prop:TakqH}, we have $d^{\langle 4 \rangle, -} (\lambda  ) =0 $, $d^{\langle 4 \rangle, +} (\lambda +N ) = 0 $ and
\begin{align}
& c_1  d^{\langle 4 \rangle, -} (\lambda + 1  ) + c_0 (d^{\langle 4 \rangle, 0} (\lambda ) -E) =0, \\
& c_n d^{\langle 4 \rangle, -} (\lambda +n ) + c_{n-1} (d^{\langle 4 \rangle, 0} (\lambda +n -1 ) -E) +c_{n-2} d^{\langle 4 \rangle, +} (\lambda +n -2 ) =0 , \quad (2 \leq n \leq N) , \nonumber \\
& c_{N} (d^{\langle 4 \rangle, 0} (\lambda +N ) -E) +c_{N-1} d^{\langle 4 \rangle, +} (\lambda +N -1 ) =0 .\nonumber 
\end{align}
From now on, we investigate the case $\lambda =\lambda _1$.
If $(\beta =) \lambda _2 - \lambda _1 \not \in \{ 1,2,\dots ,N \}  $, then we have $d^{\langle 4 \rangle, -} (\lambda _1 +n ) \neq 0$ for $n =1,2, \dots , N $ and the coefficients $c_n $ ($n =1,2, \dots , N $) are determined recursively.
If we regard $E$ as an indeterminant and set $c_0=1$, then $c_n$ is a polynomial of $E $ of degree $n$ and we denote it by $c_n(E)$.
Then we have
\begin{align}
& c_n (E) t_1 t_2 q^{h_1 +h_2 } ( 1 - q^{n }) ( 1 - q^{n -\beta }) \label{eq:rec}  \\
& - c_{n-1} (E)  [E q^{n -1 +\lambda _1} + q^{1/2}(q^{h_1 } t_1 +q^{h_2 } t_2 ) +  (q^{l_1 } t_1 +q^{l_2 } t_2 ) q^{2(n -1 +\lambda _1) +\alpha _1 +\alpha _2 -1/2 }]  \nonumber \\
& + c_{n-2} (E) q (1 - q^{n -2 +\lambda _1 +\alpha _1})(1 - q^{n-2 +\lambda _1 + \alpha _2}) =0 , \nonumber 
\end{align}
for $n=1,2,\dots ,N$, where $c_{-1} (E) =0$ and $c_{0} (E)=1$.
We define the polynomial $c_{N+1} (E)$ by
\begin{align}
& c_{N+1} (E) t_1 t_2 q^{h_1 +h_2 } ( 1 - q^{N+1 }) ( 1 - q^{N+1 -\beta }) \label{eq:recN+1}  \\
& - c_{N} (E)  [E q^{N +\lambda _1} + q^{1/2}(q^{h_1 } t_1 +q^{h_2 } t_2 ) +  (q^{l_1 } t_1 +q^{l_2 } t_2 ) q^{2(N +\lambda _1) +\alpha _1 +\alpha _2 -1/2 }]  \nonumber \\
& + c_{N-1} (E) q (1 - q^{N -1 +\lambda _1 +\alpha _1})(1 - q^{N -1 +\lambda _1 + \alpha _2}) =0 \nonumber 
\end{align}
in the case $N+1 -\beta \neq 0 $ and
\begin{align}
& c_{N+1} (E) = c_{N} (E)  [E q^{N +\lambda _1} + q^{1/2}(q^{h_1 } t_1 +q^{h_2 } t_2 ) +  (q^{l_1 } t_1 +q^{l_2 } t_2 ) q^{2(N +\lambda _1) +\alpha _1 +\alpha _2 -1/2 }] \label{eq:recN+1=0}  \\
& - c_{N-1} (E) q (1 - q^{N -1 +\lambda _1 +\alpha _1})(1 - q^{N -1 +\lambda _1 + \alpha _2}) \nonumber 
\end{align}
in the case $N+1 -\beta = 0 $.

\begin{prop} \label{prop:prop0}
Let $\lambda _1$ be the value in Eq.(\ref{eq:A4la1la2}), $\alpha \in \{  \alpha _1 , \alpha _2 \}$ and assume that $-\lambda _1 - \alpha (:=N)$ is a non-negative integer and $\beta \not \in \{ 1,2,\dots ,N\}$.
Set $c_{-1}(E)=0 $,  $c_0(E)=1 $ and we determine the polynomials $c_n(E)$ $(n=1,\dots ,N)$ recursively by Eq.(\ref{eq:rec}) 
Assume that $E=E_0$ is a solution of the algebraic equation
\begin{equation}
c_{N+1} (E)=0 ,
\label{eq:cN+1}
\end{equation}
(see Eqs.(\ref{eq:recN+1}), (\ref{eq:recN+1=0})).
Then the $q$-Heun equation has a non-zero solution of the form
\begin{equation}
f(x)= x^{\lambda _1} \sum _{n=0}^{N } c_n (E_0)  x^n .
\label{eq:gxxlapol}
\end{equation}
\end{prop}
\begin{proof}
The condition that the function $f(x)$ in Eq.(\ref{eq:gxxlapol}) satisfies $q$-Heun equation is equivalent to Eq.(\ref{eq:rec}) substituted by $E=E_0$ for $n=1,2,\dots ,N$ and 
\begin{align}
& c_{N} (E_0)  [q^{1/2}(q^{h_1 } t_1 +q^{h_2 } t_2 ) + E q^{N +\lambda _1} + (q^{l_1 } t_1 +q^{l_2 } t_2 ) q^{2(N +\lambda _1) +\alpha _1 +\alpha _2 -1/2 } \label{eq:cNcN-1cond} \\
& - c_{N-1} (E_0) q (1 - q^{N-1 +\lambda _1 +\alpha _1})(1 - q^{N-1 +\lambda _1 + \alpha _2}) =0.  \nonumber
\end{align}
Then Eq.(\ref{eq:rec}) substituted by $E=E_0$ for $n=1,2,\dots ,N$ follows from the definition of the polynomials $c_n(E)$ $(n=1,\dots ,N)$, and Eq.(\ref{eq:cNcN-1cond}) follows from $c_{N+1} (E_0)=0  $ by Eqs.(\ref{eq:recN+1}), (\ref{eq:recN+1=0}).
\end{proof}
We call $f(x)$ in Eq.(\ref{eq:gxxlapol}) a polynomial-type solution, which is a product of $x^{\lambda _1} $ and a polynomial.
If the accessory parameter $E$ of the $q$-Heun equation satisfies the equation $c_{N+1} (E)=0$, then the $q$-Heun equation has a polynomial-type solution.
We call $c_{N+1} (E)$ the spectral polynomial of the $q$-Heun equation.

\section{Real root property of the spectral polynomial of the $q$-Heun equation} \label{sec:realroot}

We may use the theory of Sturm sequence from the three term relations for $c_n$ and we obtain real root property of the spectral polynomial $c_{N+1} (E) $.
The following lemma is obtained by the argument of Sturm sequence, which was essentially applied to Lam\'e equation and Heun equation in \cite{WW,Tak1,CKLT2}.
\begin{lemma} \label{lem:dist}
Let $N$ be a non-negative integer.
Assume that $d_n > 0 $ and $d'_{n+1} >0 $ for $n=1, \dots ,N$ and $p_n >0$ and $q_n \in \Rea $ for $n=1, \dots ,N+1$. 
Set $c_{-1}(E)=0$, $c_0 (E)=1$, and determine the polynomial $c_n (E)$ $(n=1,2,\dots ,N+1)$ recursively by
\begin{equation}
d_{n} c_n (E) = (p_n E +q_n ) c_{n-1} (E) - d'_n c_{n-2} (E) .
\label{eq:lem3termrel}
\end{equation}
In the case $ d_{N+1} = 0$, we set $c_{N+1} (E) = (p_{N+1} E +q_{N+1} ) c_{N} (E) - d'_{N+1} c_{N-1} (E)  $.
Then the polynomial $c_n (E)$ $(n=1,2,\dots ,N+1)$ has $n$ real distinct zeros $s_{j}^{(n)}$ $(j=1,\dots, n)$ such that
\begin{equation}
s_{1}^{(n)}<s_{1}^{(n-1)}<s_{2}^{(n)}<s_{2}^{(n-1)}<\cdots<s_{n-1}^{(n)}<s_{n-1}^{(n-1)}<s_{n}^{(n)}
\end{equation}
for $n=2,\dots ,N+1$. 
\end{lemma}
\begin{proof}
It follows from the assumption that $c_n(E)$ is a polynomial of degree $n$ such that the coefficient of $E^n$ is positive.
The polynomial $c_1(E)= (p_{1} E +q_{1} )/d_1  $ has one real zero.
We show that if $r\in \{ 1,2,\dots ,N-1 \} $ and the polynomials $c_{r-1} (E)$ and $c_{r} (E) $ has real distinct zeros such that 
\begin{equation}
s_{1}^{(r)}<s_{1}^{(r-1)}<s_{2}^{(r)}<s_{2}^{(r-1)}<\cdots<s_{r-1}^{(r)}<s_{r-1}^{(r-1)}<s_{r}^{(r)},
\label{eq:lemrr-1}
\end{equation}
then the polynomial $c_{r+1} (E)$ has $r+1$ real distinct zeros such that 
\begin{equation}
s_{1}^{(r+1)}<s_{1}^{(r)}<s_{2}^{(r+1)}<s_{2}^{(r)}<\cdots<s_{r}^{(r+1)}<s_{r}^{(r)}<s_{r+1}^{(r+1)}.
\label{eq:lemr+1r}
\end{equation}
Since $c_{r-1} (s_{r-1}^{(r-1)}) =0  $, $s_{r-1}^{(r-1)}<s_{r}^{(r)} $ and $c_{r-1} (E) \to +\infty $ as $E \to +\infty $, we have $c_{r-1} (s_{r}^{(r)}) >0  $.
Moreover it follows from Eq.(\ref{eq:lemrr-1}) and $c_{r-1} (s_{r-j}^{(r-1)}) =0 $ $(j=1,2, \dots ,r-1)$  that $c_{r-1} (s_{r-1}^{(r)}) <0$, $c_{r-1} (s_{r-2}^{(r)}) >0 , \dots $, i.e.~$(-1)^{j }c_{r-1} (s_{r-j}^{(r)}) >0$ $(j=1,2,\dots ,r-1)$. 
By $c_{r-1} (s_{r-j}^{(r)}) =0 $ and Eq.(\ref{eq:lem3termrel}), we obtain $d_{r+1} c_{r+1} (s_{r-j}^{(r)}) =  - d'_{r+1} c_{r-1} (s_{r-j}^{(r)})$.
Therefore $(-1)^{j +1}c_{r+1} (s_{r-j}^{(r)}) >0$ for $j=0,1,\dots ,r-1$, which follows from the assumption $d_{r+1} >0$ and $d'_{r+1} >0$.
Since $c_{r+1}(s_{r}^{(r)})<0 $ and $c_{r+1} (E) \to +\infty $ as $E \to +\infty $, there exists a real number $s_{r+1}^{(r+1)} $ such that $s_{r}^{(r)}<s_{r+1}^{(r+1)} $ and $c_{r+1} (s_{r+1}^{(r+1)}) =0 $.
It follows from $c_{r+1} (s_{r-j}^{(r)}) c_{r+1} (s_{r-j+1}^{(r)}) <0$ and the intermediate value theorem that there exists a real number $s_{r-j+1}^{(r+1)} $ such that $s_{r-j}^{(r)}<s_{r-j+1}^{(r+1)} < s_{r-j+1}^{(r)}$ and $c_{r+1} (s_{r-j+1}^{(r+1)}) =0 $ for $j=1,2, \dots ,r-1$.
It also follows from $(-1)^r c_{r+1}(s_{1}^{(r)})>0 $ and $c_{r+1} (E) \to (-1)^{r+1} \infty $ as $E \to -\infty $ that there exists a real number $s_{1}^{(r+1)} $ such that $s_{1}^{(r+1)} < s_{1}^{(r)}$ and $c_{r+1} (s_{1}^{(r+1)}) =0 $.
Therefore the polynomial $c_{r+1} (E)$ has $r+1$ real distinct zeros $s_{j}^{(r+1)} $ $(j=1,2,\dots ,r+1) $ which satisfy Eq.(\ref{eq:lemr+1r}).

It remains to be shown that $c_{N+1} (E)$ has $N+1$ real distinct zeros $s_{j}^{(N+1)} $  $(j=1,2,\dots ,N+1) $ which satisfy Eq.(\ref{eq:lemr+1r}) for $r=N$.
If $d_{N+1} >0$, then it is shown by applying the previous discussion.
If $d_{N+1} =0$, then $c_{N+1} (E)$ is defined separately and it is reduced to the case $d_{N+1} =1$.
If $d_{N+1} <0$, then we set $\tilde{d}_{N+1} = -d_{N+1} $ and $\tilde{c}_{N+1} (E)= -c_{N+1} (E)$.
Since $\tilde{d}_{N+1} >0 $, it is shown that $\tilde{c}_{N+1} (E)$ has $N+1$ real distinct zeros $s_{j}^{(N+1)} $  $(j=1,2,\dots ,N+1) $ which satisfy Eq.(\ref{eq:lemr+1r}) for $r=N$.
Obviously the zeros of $c_{N+1} (E)$ coincide with those of $\tilde{c}_{N+1} (E)$.
\end{proof}
In Lemma \ref{lem:dist} we obtained that if $d_n > 0 $ and $d'_{n+1} >0 $ for $n=1, \dots ,N$ and $p_{n} >0$ for $n=1, \dots ,N+1$, then the polynomial $c_{N+1} (E)$ has $N+1$ real distinct zeros.
We apply the lemma for the three term relation in Eq.(\ref{eq:rec}).
\begin{thm} \label{thm:SturmqHeun}
Assume that $N= -\lambda _1 -\alpha _1 $ is a non-negative integer, $0<q<1 $ and $t_1$, $t_2$, $h_1$, $h_2$, $l_1$, $l_2$, $\alpha _1$, $\alpha _2$, $\beta  $ are all real.
If one of the following conditions is satisfied, then the equation $c_{N+1} (E)=0$ in Eq.(\ref{eq:cN+1}) has all its roots real and distinct.\\
(i) $t_1 t_2 >0$, $\alpha _2-\alpha _1<1 $ and $\beta < 1$.\\
(ii) $t_1 t_2 >0$, $\alpha _2-\alpha _1>N $ and $ \beta > N$.\\
(iii) $t_1 t_2 <0$, $ \alpha _2-\alpha _1>N $ and $\beta < 1$.\\
(iv) $t_1 t_2 <0$, $ \alpha _2-\alpha _1 <1 $ and $\beta > N$.
\end{thm}
\begin{proof}
The theorem is shown by applying Lemma \ref{lem:dist} to Eqs.(\ref{eq:rec}), (\ref{eq:recN+1}) and (\ref{eq:recN+1=0}).
As for (i), it follows from the condition of (i) that $n + \lambda _1 +\alpha _2 <1 $ for $n=1,2,\dots ,N$ and $n-\beta >0$ for $n=1,2,\dots $.
Hence the assumption of Lemma \ref{lem:dist} is confirmed.
(iv) is shown similarly.
(ii) and (iii) follows from the lemma by multiplying $-1$.
\end{proof}

\section{Analysis of the spectral polynomial by the ultradiscrete limit} \label{sec:UDL}

In the previous sections, it was shown that the condition for the eigenvalue $E$ such that $q$-Heun equation admits a non-zero polynomial solution is described by the roots of the spectral polynomial $c_{N+1} (E)$.
However it is not possible to solve the algebraic equation $c_{N+1} (E)=0 $ explicitly.
In this section, we investigate the solution of $c_{N+1} (E)=0 $ by the ultradiscrete limit $q\to +0$.
Detailed properties on convergence will be discussed in the appendix.

We define the equivalence of functions of the variable $q$ by 
\begin{equation}
a(q) \sim b (q) \; \Leftrightarrow \; \lim _{q \to +0} \frac{a(q)}{b(q)} =1.
\end{equation}
We also define the equivalence $\sum _{j=0}^{M} a_j(q) E^j \sim \sum _{j=0}^{M} b_j (q) E^j $ of the polynomials of the variable $E$ by $ a_j (q) \sim b _j(q) $ for $j=0,\dots ,M$.

We investigate solutions of $q$-Heun equation in the form 
\begin{equation}
f(x)= x^{\lambda _1} \sum _{n=0}^{N } c_n (E)  x^n ,
\end{equation}
where $\lambda _1 = (h_1 +h_2 -l_1-l_2 -\alpha _1-\alpha _2 -\beta +2)/2 $ under the condition of Theorem \ref{thm:SturmqHeun} (i).
For simplicity, we assume that $t_1>0 , \; t_2>0, \; h_1 <h_2 $ and $l_1<l_2.$
From now on, we assume that 
\begin{equation}
N= -\lambda _1 -\alpha _1 \in \Zint _{\geq 0}, \; \beta <1, \; \alpha _2-\alpha _1<1, \; t_1>0 , \; t_2>0, \; h_1 <h_2 , \; l_1<l_2.
\label{eq:assump}
\end{equation} 
Recall that the polynomials $c_{n} (E)$ $(n=1,2,\dots ,N)$ are determined recursive by Eq.(\ref{eq:rec}), i.e. 
\begin{align}
& c_n (E) t_1 t_2 [ q^{h_1 +h_2 } ( 1 - q^{n }) ( 1 - q^{n -\beta }) ] \label{eq:rec01}  \\
& = c_{n-1} (E)  [E q^{n -1 +\lambda _1} + q^{1/2}(q^{h_1 } t_1 +q^{h_2 } t_2 ) +  (q^{l_1 } t_1 +q^{l_2 } t_2 ) q^{2(n -1 +\lambda _1) +\alpha _1 +\alpha _2 -1/2 }]  \nonumber \\
& - c_{n-2} (E) [ q (1 - q^{n -2 +\lambda _1 +\alpha _1})(1 - q^{n-2 +\lambda _1 + \alpha _2}) ] , \nonumber 
\end{align}
with the initial condition $c_0(E)=1$ and $c_{-1} (E)=0$.
The spectral polynomial $c_{N+1} (E)$ is determined by setting $n=N+1$ in Eq.(\ref{eq:rec01}).
As $q\to +0$,
\begin{align}
& ( 1 - q^{n }) ( 1 - q^{n- \beta }) \sim 1 ,  \; (n=1,2,\dots ) \\
& q^{-h_1 -h_2 +1}  (1 - q^{n -2 +\lambda _1 +\alpha _1})(1 - q^{n-2 +\lambda _1 + \alpha _2}) \sim q^{2n-1 -l_1-l_2 -\beta } , \; (n=2,3,\dots , N+1).\nonumber 
\end{align}
Combining with the condition $h_1 <h_2 $ and $l_1<l_2$, we have
\begin{align}
c_n (E) \sim & \: t_1^{-1} t_2^{-1} [  E q^{n -1 -h_1 -h_2 +\lambda _1} + t_1 q^{1/2 -h_2 } + t_1 q^{ 2n -1/2 -l_2 - \beta  }] c_{n-1} (E) \label{eq:rec00}   \\
&  - t_1^{-1} t_2^{-1} q^{2n-1 -l_1-l_2 -\beta } c_{n-2} (E) \nonumber
\end{align}
for $ n=1,2,\dots ,N+1$ under the assumption that there are no cancellation of the leading terms of the coefficients of $E^j$ $(j=0,1,\dots ,n-1)$ in the right hand side with respect to the limit $q \to +0$.

If $1 +h_2 -l_2 - \beta >0 $ (resp. $2N+1 +h_2 -l_2 -\beta <0$), then we have $q^{1/2 -h_2 }  + q^{ 2n -1/2 -l_2 - \beta  } \sim q^{2n -1/2 -l_2 - \beta  } $ (resp. $q^{1/2 -h_2 }  + q^{ 2n -1/2-l_2 - \beta } \sim q^{1/2 -h_2 } $) for $n=1,2, \dots , N+1$.
We investigate the behaviour of the spectral polynomial $c_{N+1}(E)$ and the associated solutions of $q$-Heun equation as $q\to +0$ for the case $1 +h_2 -l_2 - \beta >0 $ or  $2N+1 +h_2 -l_2 -\beta <0$. 

\subsection{The case $1 +h_2 -l_2 - \beta >0 $} $ $

If $1 +h_2 -l_2 - \beta >0 $, then we have
\begin{align}
& c_n (E) \sim t_1^{-1} t_2^{-1} (  E q^{n -1 -h_1 -h_2 +\lambda _1} + t_1 q^{1/2 -h_2 } ) c_{n-1} (E)  - t_1^{-1} t_2^{-1} q^{2n-1 -l_1-l_2 -\beta } c_{n-2} (E) \label{eq:rec0101}  
\end{align}
for $ n=1,2,\dots ,N+1$ under the assumption that there are no cancellation of the leading terms of the coefficients of $E^j$ $(j=0,1,\dots ,n-1)$ in the right hand side with respect to the limit $q \to +0$.
Since $c_0 (E) =1$ and $c_{-1} (E)=0$, we have $c_1 (E) \sim t_1^{-1} t_2^{-1} ( E q^{-h_1 -h_2 +\lambda _1} + t_1 q^{1/2 -h_2 } )$ and
\begin{align}
& c_2 (E) \sim t_1^{-2} t_2^{-2} ( E q^{1 -h_1 -h_2 +\lambda _1} + t_1 q^{1/2 -h_2 } )( E q^{-h_1 -h_2 +\lambda _1} + t_1 q^{1/2 -h_2 } ) - t_1^{-1} t_2^{-1} q^{3 -l_1-l_2 -\beta }.
\label{eq:c2Ev1}
\end{align}
If $1-2h_2 \neq 3-l_1 -l_2 -\beta $, i.e. $2+ 2 h_2  - l_1 - l_2 -\beta \neq 0 $, then there are no cancellation of the leading terms.
If $2+ 2 h_2  - l_1 - l_2 -\beta >0 $, then we may ignore the term $t_1^{-1} t_2^{-1} q^{3 -l_1-l_2 -\beta } $ and we have
\begin{align}
& c_2 (E) \sim t_1^{-2} t_2^{-2}  ( E q^{1 -h_1 -h_2 +\lambda _1} + t_1 q^{1/2 -h_2 }) ( E q^{-h_1 -h_2 +\lambda _1} + t_1 q^{1/2 -h_2 } ) .
\end{align}
Moreover we can obtain the following proposition;
\begin{prop} \label{prop:cncn-11}
If $1 +h_2 -l_2 - \beta >0 $ and $2+ 2 h_2  - l_1 - l_2 -\beta >0  $ then we have
\begin{align}
& c_n (E) \sim t_1^{-1} t_2^{-1} ( E q^{n -1 -h_1 -h_2 +\lambda _1} + t_1 q^{1/2 -h_2 } ) c_{n-1} (E) \label{eq:rec010101}  
\end{align}
for $n=1,2,\dots ,N+1$.
\end{prop}
\begin{proof}
Let $k \in \{ 1,2,\dots ,N\}$ and assume that Eq.(\ref{eq:rec010101}) holds for $n\leq k $.
Then we have
\begin{align}
& t_1^{-1} t_2^{-1} ( E q^{k -h_1 -h_2 +\lambda _1} + t_1 q^{1/2 -h_2 } ) c_{k} (E)  - t_1^{-1} t_2^{-1} q^{2k +1 -l_1-l_2 -\beta } c_{k-1 } (E) \\
&  \sim  \{ t_1^{-2} t_2^{-2} (  E q^{k -h_1 -h_2 +\lambda _1} + t_1 q^{1/2 -h_2 })( E q^{k -h_1 -h_2 +\lambda _1} + t_1 q^{1/2 -h_2 } ) \nonumber \\
& \qquad - t_1^{-1} t_2^{-1} q^{2k +1 -l_1-l_2 -\beta } ]\} c_{k-1 } (E) . \nonumber 
\end{align}
Since $1-2h_2 < 3-l_1 -l_2 -\beta \leq 2k +1 -l_1-l_2 -\beta $, we may neglect the term $t_1^{-1} t_2^{-1} q^{2k +1 -l_1-l_2 -\beta } c_{k-1 } (E)  $ and we have  Eq.(\ref{eq:rec010101}) for $n= k +1$.
\end{proof}
\begin{thm} \label{thm:specpol1}
We assume Eq.(\ref{eq:assump}), $1 +h_2 -l_2 - \beta >0 $ and $2+ 2 h_2  - l_1 - l_2 -\beta >0  $.\\
(i) The spectral polynomial $c_{N+1} (E)$ satisfies
\begin{align}
 c_{N+1} (E) \sim & (t_1 t_2 )^{-N-1} q^{(N/2 + \lambda _1 -h_1 -h_2)(N+1) } ( E + q^{1/2 -N +h_1 -\lambda _1} t_1 )( E + q^{3/2 -N +h_1 -\lambda _1} t_1 ) \label{eq:specpoly01} \\
& \cdots ( E + q^{-1/2 +h_1 -\lambda _1 } t_1 )( E + q^{1/2 +h_1 -\lambda _1 } t_1 ) . \nonumber 
\end{align}
(ii) There exist solutions $E_k (q)$ $(k=1,2,\dots ,N+1)$ to the equation $c_{N+1} (E) =0 $ for sufficiently small $q$ such that
\begin{align}
& E _k(q)  \sim -q^{-k + 3/2 +h_1 -\lambda _1  } t_1 .
\label{eq:Ekqcase1}
\end{align}
\end{thm}
\begin{proof}
We obtain (i) by applying Proposition \ref{prop:cncn-11} repeatedly.
(ii) follows from Corollary \ref{cor:sol}.
\end{proof}
\begin{remk}
If  $2+ 2 h_2  - l_1 - l_2 -\beta <0  $, then the zeros of the polynomials $c_n(E)$ have a different feature.
It follows from Eq.(\ref{eq:c2Ev1}) and the assumption $2 h_2  - l_1 - l_2 -\beta <-2 $ that 
\begin{align}
& c_2 (E) \sim t_1^{-2} t_2^{-2} ( q^{1 -2h_1 -2h_2 +2 \lambda _1} E^2 + t_1 q^{1/2 -h_2 } q^{-h_1 -h_2 +\lambda _1} E  - t_1 t_2 q^{3 -l_1-l_2 -\beta } ) .
\end{align}
The spectral polynomial $c_{N+1} (E)$ would be different from Eq.(\ref{eq:specpoly01}).
See \cite{Koj,Sat} for details.
\end{remk}

We investigate the polynomial solution of $q$-Heun equation for the case $1 +h_2 -l_2 - \beta >0 $  and $2+ 2 h_2  - l_1 - l_2 -\beta >0  $.
Then the spectral polynomial $c_{N+1 } (E) =0$ has solutions which satisfy Eq.(\ref{eq:Ekqcase1}).
Write the normalized polynomial solution of the $q$-Heun equation as 
\begin{equation}
x^{\lambda _1} \sum _{n=0}^{N } c_n (E_k )  x^n ,
\label{eq:solqHEj1}
\end{equation}
where $\lambda _1 = (h_1 +h_2 -l_1-l_2 -\alpha _1-\alpha _2 -\beta +2)/2 $, $c_0 (E_k )=1$ and $E_k$ is an abbreviation of $E_k (q)$.
\begin{thm} \label{thm:eigenf1}
Let $k\in \{ 1,2,\dots ,N+1 \}$.
Assume Eq.(\ref{eq:assump}), $1 +h_2 -l_2 - \beta >0 $, $2+ 2 h_2  - l_1 - l_2 -\beta >0  $ and the value $E=E_k $ is a solution of the characteristic equation $c_{N+1} (E)=0$ such that  $E_k  \sim -q^{3/2 -k +h_1 -\lambda _1} t_1 $.\\
(i) The coefficients of the normalized polynomial solution in Eq.(\ref{eq:solqHEj1}) satisfy
\begin{align}
& c_n (E_k) \sim  - q^{n -k +1/2 -h_2 } t_2^{-1} c_{n-1} (E_k), & 1 \leq n \leq k-1, \\
& c_{n} (E_k) \sim  q^{2n +1/2 +h_2 -l_1 -l_2 -\beta } t_1^{-1} c_{n-1} (E_k), & k \leq n \leq N . \nonumber
\end{align}
(ii) We have
\begin{align}
& \sum _{n=0}^{N } c_n (E_k)  x^n \sim \prod _{j=1}^{k-1} ( 1-  q^{j -k +1/2 -h_2 } t_2^{-1} x ) \prod _{j=k}^N (1 +  q^{2j +1/2 +h_2 -l_1 -l_2 -\beta } t_1^{-1} x).
\end{align}
(iii) There exists $q_j \in \Rea _{>0}$ for $j=1,2,\dots ,M$ such that the polynomial $\sum _{n=0}^{N } c_n (E_k)  x^n  $ has a zero $x= x_j(q) $ for $0<q<q_j$ which is continuous on $q$ and satisfies 
\begin{align}
& \lim _{q \to +0} \frac{x _j (q)}{q^{j -1/2 +h_2 }t_2 } =1 , \; (j=1,\dots ,k-1), \\
& \lim _{q \to +0} \frac{x _j (q)}{- q^{-2j -1/2 -h_2 +l_1 +l_2 +\beta }t_1 } =1 ,\; (j=k , \dots ,N). \nonumber
\end{align}
\end{thm}
\begin{proof}
Since Eq.(\ref{eq:solqHEj1}) is a solution of $q$-Heun equation, the coefficients $c_n (E_k) $ satisfies
\begin{align}
 c_n (E_k) \sim & t_1^{-1} t_2^{-1} ( -t_1 q^{3/2 -k +h_1 -\lambda _1} q^{n -1 -h_1 -h_2 +\lambda _1} + t_1 q^{1/2 -h_2 } ) c_{n-1} (E_k) , \label{eq:rec0101Ej} \\
&  - t_1^{-1} t_2^{-1} q^{2n-1 -l_1-l_2 -\beta } c_{n-2} (E_k) . \nonumber  
\end{align}
for $n =1, \dots ,N+1$, if the leading terms on the right hand side are not cancelled.
Note that $c_{-1} (E_k)=0$, $c_{N+1}(E_k)=0$ and $c_{0} (E_k)=1 $. 
If $1 \leq n \leq k-1$, then the term $t_2 ^{-1} q^{1/2 -h_2 } c_{n-1} (E_k)$ does not affect the leading term and the right hand side of Eq.(\ref{eq:rec0101Ej}) is equivalent to
\begin{align}
 -t_2^{-1} q^{1/2 +n -k  -h_2 } c_{n-1} (E_k) - t_1^{-1} t_2^{-1} q^{2n-1 -l_1-l_2 -\beta } c_{n-2} (E_k) .  
\end{align}
We now show that $c_n (E_k) \sim  -t_2^{-1} q^{1/2 -h_2  + n -k }  c_{n-1} (E_k)$ for $1 \leq n \leq k-1$ and $k\geq 2$.
It follows from Eq.(\ref{eq:rec0101Ej}) for the case $n=1$ that $c_1 (E_k) \sim  -t_2^{-1} q^{3/2  -k  -h_2 } $.
If $c_{n-1} (E_k) \sim  -t_2^{-1} q^{1/2 -h_2  + n -1 -k }  c_{n-2} (E_k)$, then it follows from $2+ 2 h_2  - l_1 - l_2 -\beta >0 $ that the term $q^{1/2 +n -k  -h_2 }  q^{1/2 -h_2  + n -1 -k }$ is stronger than $ q^{2n-1 -l_1-l_2 -\beta } $.
Hence we may ignore the term $t_1^{-1} t_2^{-1} q^{2n-1 -l_1-l_2 -\beta } c_{n-2} (E_k) $, the leading term of the right hand side is contained in $-t_2^{-1} q^{1/2 -h_2  + n -k }  c_{n-1} (E_k) $  and we have  $c_n (E_k) \sim  -t_2^{-1} q^{1/2 -h_2  + n -k }  c_{n-1} (E_k)$ for $1 \leq n \leq k-1$.

We show that $c_{n} (E_k) \sim t_1^{-1} q^{1/2 +2n +h_2 -l_1 -l_2 -\beta } c_{n-1} (E_k)$ for $k \leq n \leq N$.
It follows from the three term relation that
\begin{align}
&  c_{n-2} (E_k)  \sim t_1 q^{3/2 -2n -h_2 +l_1 +l_2 +\beta } c_{n-1} (E_k)  -  t_1 t_2 q^{-2n +1 +l_1 +l_2 +\beta }c_n (E_k) 
\label{eq:reccn-210}
\end{align}
for $k+1 \leq n \leq N+1$, if the leading terms on the right hand side are not cancelled.
In the case $n=N+1$, we have $c_{N} (E_k) \sim t_1^{-1} q^{1/2 +2N +h_2 -l_1 -l_2 -\beta } c_{N-1} (E_k) $ by Eq.(\ref{eq:reccn-210}).
We assume that $c_{n} (E_k) \sim t_1^{-1} q^{1/2 +2n +h_2 -l_1 -l_2 -\beta } c_{n-1} (E_k) $ for some $n$ such that $k+1 \leq n \leq N $.
Since $q^{3/2 -2n -h_2 +l_1 +l_2 +\beta } q^{1/2 +2n +h_2 -l_1 -l_2 -\beta }$ is stronger than $  q^{-2n +1 +l_1 +l_2 +\beta }$, the term $ t_1 t_2 q^{-2n +1 +l_1 +l_2 +\beta }c_n (E_k) $ in the right hand side of Eq.(\ref{eq:reccn-210}) is negligible and we have $c_{n-1} (E_k) \sim  t_1^{-1} q^{-3/2 +2n +h_2 -l_1 -l_2 -\beta } c_{n-2} (E_k)$.
Thus we have shown $c_{n} (E_k) \sim t_1^{-1} q^{1/2 +2n +h_2 -l_1 -l_2 -\beta } c_{n-1} (E_k)$ for $k \leq n \leq N$.
Therefore we obtain (i).

We show (ii).
Write
\begin{align}
& \prod _{j=1}^{N} ( 1+ s_j x ) = \sum _{n=0}^{N } d_n  x^n , \; s_j= \left\{ 
\begin{array}{ll}
t_2^{-1} q^{j -k + 1/2 -h_2  }, & j=1,\dots ,k -1 ,\\
- t_1^{-1} q^{2j + 1/2 +h_2 -l_1 -l_2 -\beta }, & j=k,\dots ,N. 
\end{array}
\right.
\end{align}

Then it follows from $-1/2 -h_2 < 3/2 +h_2 -l_1 -l_2 -\beta < 2k +1/2 +h_2 -l_1 -l_2 -\beta $ that $s_j$ is stronger than $s_{j+1}$ for $j=1,\dots ,N-1$.
Hence the assumption of Theorem \ref{thm:zerosasymp} is confirmed and we have $d_n \sim \prod _{j=1}^n s_j $ for $n=1,\dots ,N$. 
On the other hand, it follows from (i) that $c_{n+1} (E_k) \sim s_{n+1} c_n (E_k) $.
Combining with $d_0=1=c_{0} (E_k)$, we obtain (ii).

(iii) follows from Theorem \ref{thm:zerosasymp}.
\end{proof}

\subsection{The case $2N+1 +h_2 -l_2 -\beta <0 $} $ $

If $2N+1 +h_2 -l_2 -\beta <0 $, then 
\begin{align}
 c_n (E) \sim & t_1 ^{-1} t_2 ^{-1} ( E q^{n -1 -h_1 -h_2 +\lambda _1} + q^{2n -1/2 -l_2 - \beta  } t_1 ) c_{n-1} (E)  \label{eq:rec0102} \\
 &  - q^{2n-1 -l_1-l_2 -\beta } t_1 ^{-1} t_2 ^{-1} c_{n-2} (E) \nonumber
\end{align}
for $n=1,2,\dots ,N+1$ under the assumption of no cancellation as in the case $1 +h_2 -l_2 - \beta >0 $. 
Then we have $ c_1 (E) \sim t_1 ^{-1} t_2 ^{-1} ( E q^{ -h_1 -h_2 +\lambda _1} + q^{3/2 -l_2 - \beta  } t_1 ) $ and
\begin{align}
 c_2 (E) \sim & t_1 ^{-2} t_2 ^{-2} ( E q^{1 -h_1 -h_2 +\lambda _1} + q^{7/2 -l_2 - \beta  } t_1 )( E q^{ -h_1 -h_2 +\lambda _1} + q^{3/2 -l_2 - \beta  } t_1 ) \\
&  - q^{3 -l_1-l_2 -\beta } t_1 ^{-1} t_2 ^{-1} . \nonumber 
\end{align}
If  $2+ l_1 - l_2 -\beta <0 $, then we may ignore the term $t_1^{-1} t_2^{-1} q^{3 -l_1-l_2 -\beta } $ and we have
\begin{align}
& c_2 (E) \sim t_1 ^{-2} t_2 ^{-2} ( E q^{1 -h_1 -h_2 +\lambda _1} + q^{7/2 -l_2 - \beta  } t_1 )( E q^{ -h_1 -h_2 +\lambda _1} + q^{3/2 -l_2 - \beta  } t_1 ).
\end{align}
Moreover we can obtain the following proposition
\begin{prop} \label{prop:cncn-12}
Let $k \in \{ 1,2,\dots ,N \}$.
If $2N+1 +h_2 -l_2 -\beta <0  $ and $2k + l_1 - l_2 -\beta <0   $ then $c_n(E)$ satisfies
\begin{align}
& c_n (E) \sim t_1 ^{-1} t_2 ^{-1} ( E q^{n -1 -h_1 -h_2 +\lambda _1} + q^{2n -1/2 -l_2 - \beta  } t_1 ) c_{n-1} (E)
\end{align}
for $n =1,2,\dots , k+1$.
\end{prop}
\begin{proof}
Assume that $c_{n-1} (E) \sim t_1 ^{-1} t_2 ^{-1} ( E q^{n -2 -h_1 -h_2 +\lambda _1} + q^{2n -5/2 -l_2 - \beta  } t_1 ) c_{n-2} (E)$.
Then we have
\begin{align}
 c_n (E) \sim & \{ t_1 ^{-2} t_2 ^{-2} ( E q^{n -1 -h_1 -h_2 +\lambda _1} + q^{2n -1/2 -l_2 - \beta  } t_1 )( E q^{n -2 -h_1 -h_2 +\lambda _1} + q^{2n -5/2 -l_2 - \beta  } t_1 ) \\
& - q^{2n-1 -l_1-l_2 -\beta } t_1 ^{-1} t_2 ^{-1} \} c_{n-2} (E) . \nonumber
\end{align}
Then $ 4n -3 -2 l_2 - 2 \beta < 4n -3 - 2k - l_1 - l_2 - \beta \leq 2n -1 - l_1   - l_2 - \beta $ and we may ignore the term $q^{2n-1 -l_1-l_2 -\beta } t_1 ^{-1} t_2 ^{-1}  c_{n-2} (E) $.
\end{proof}
Therefore, if $2N+1 +h_2 -l_2 -\beta <0  $ and $2N + l_1 - l_2 -\beta <0 $, then $c_n(E)$ satisfies
\begin{align}
 c_n (E) & \sim q^{n -1 -h_1 -h_2 +\lambda _1} t_1 ^{-1} t_2 ^{-1} ( E  + t_1 q^{ n - 3/2 +\lambda _1 +l_1 + \alpha _1 + \alpha _2 } ) c_{n-1} (E) \label{eq:rec010201} 
\end{align}
for $n=1,2,\dots , N+1$.
As Theorem \ref{thm:specpol1} in the previous subsection, we obtain the following theorem;
\begin{thm} \label{thm:specpol2}
We assume Eq.(\ref{eq:assump}), $2N+1 +h_2 -l_2 -\beta <0 $ and $2N +l_1 -l_2 -\beta <0$.\\
(i) The spectral polynomial $c_{N+1} (E)$ satisfies
\begin{align}
& c_{N+1} (E) \sim \: (t_1 t_2 )^{-N-1} q^{(N/2 + \lambda _1 -h_1 -h_2)(N+1) } ( E + q^{N -1/2 +\lambda _1 +l_1 + \alpha _1 + \alpha _2 } t_1 )\\
& \quad ( E + q^{N- 3/2 +\lambda _1 +l_1 + \alpha _1 + \alpha _2 } t_1 )  \cdots ( E + q^{1/2 +\lambda _1 +l_1 + \alpha _1 + \alpha _2} t_1 )( E + q^{-1/2 +\lambda _1 +l_1 + \alpha _1 + \alpha _2 } t_1 ) . \nonumber 
\end{align}
(ii) There exist solutions $E_k (q)$ $(k=1,2,\dots ,N+1)$ to the equation $c_{N+1} (E) =0 $ for sufficiently small $q$ such that
\begin{align}
& E _k(q)  \sim -q^{k-3/2 + \lambda _1 + l_1 + \alpha _1 +\alpha _2 } t_1 .
\label{eq:Ekqcase2}
\end{align}
\end{thm}
We investigate the polynomial solution of $q$-Heun equation for the value $E=E_k$ such that $E_k \sim -q^{k-3/2 + \lambda _1 + l_1 + \alpha _1 +\alpha _2 } t_1 $ $(k\in \{ 1,2,\dots ,N+1\})$ in the case $2N+1 +h_2 -l_2 -\beta <0 $ and $2N +l_1 -l_2 -\beta <0  $.
Write the normalized polynomial solution as 
\begin{equation}
x^{\lambda _1} \sum _{n=0}^{N } c_n (E_k)  x^n ,
\label{eq:solqHEj2}
\end{equation}
where $c_0 (E_k)=1$.
Then we have the following theorem which can be proved similarly to Theorem \ref{thm:eigenf1}.
\begin{thm} \label{thm:eigenf2}
Let $k\in \{ 1,2,\dots ,N+1 \}$.
Assume Eq.(\ref{eq:assump}), $2N+1 +h_2 -l_2 -\beta <0 $ and $2N +l_1 -l_2 -\beta <0  $ and the value $E=E_k$ is a solution of the characteristic equation $c_{N+1} (E)=0$ such that  $E_k \sim -q^{k-3/2 + \lambda _1 + l_1 + \alpha _1 +\alpha _2 } t_1 $.\\
(i) The coefficients of the normalized polynomial solution in Eq.(\ref{eq:solqHEj2}) satisfy
\begin{align}
& c_n (E_k) \sim  t_2^{-1} q^{2n -1/2 -l_2 - \beta   }  c_{n-1} (E_k), & 1 \leq n \leq k-1, \\
& c_{n} (E_k) \sim - t_1^{-1} q^{n-k +1/2 -l_1 } c_{n-1} (E_k), & k \leq n \leq N . \nonumber
\end{align}
(ii) We have
\begin{align}
& \sum _{n=0}^{N } c_n (E_k)  x^n \sim \prod _{j=1}^{k-1} ( 1 + t_2^{-1} q^{ 2j -1/2 -l_2 - \beta   }   x ) \prod _{j=k}^N (1- t_1^{-1} q^{j-k +1/2 -l_1  } x) .
\end{align}
(iii) There exists $q_j \in \Rea _{>0}$ for $j=1,2,\dots ,M$ such that the polynomial $\sum _{n=0}^{N } c_n (E_k)  x^n  $ has a zero $x= x_j(q) $ for $0<q<q_j$ which is continuous on $q$ and satisfies 
\begin{align}
& \lim _{q \to +0} \frac{x _j (q)}{- t_2 q^{-2j +1/2 +l_2 + \beta  }} =1 , \; (j=1,\dots ,k-1), \\
& \lim _{q \to +0} \frac{x _j (q)}{t_1 q^{k -j -1/2 + l_1 }} =1 ,\; (j=k , \dots ,N). \nonumber
\end{align}
\end{thm}

\section{Concluding remarks} \label{sec:rem}

In this paper, we investigated polynomial-type solutions of the $q$-Heun equation.
We defined the spectral polynomial of the accessory parameter $E$ in the case that the parameters of the $q$-Heun equation satisfies the assumption of Proposition \ref{prop:prop0}.
The polynomial-type solution of the $q$-Heun equation exists, if the accessory parameter is a root of the spectral polynomial.
Then we obtained sufficient conditions that all the roots of the spectral polynomial is real and distinct in section \ref{sec:realroot}.
To find the behaviour of the roots of the spectral polynomial and the associated polynomial-type solution of the $q$-Heun equation, we considered the ultradiscrete limit $q \to +0$ and we obtained the behaviour of them in the case $1 +h_2 -l_2 - \beta >0 $ and $2+ 2 h_2  - l_1 - l_2 - \beta >0  $ and the case $2N+1 +h_2 -l_2 -\beta <0 $ and $2N +l_1 -l_2 -\beta <0  $.

We point out problems which should be clarified in the near future.

One problem is to consider the ultradisctere limit of the spectral polynomial and the associated polynomial-type solution of the $q$-Heun equation for the case $-2N-1 <h_2-l_2-\beta <-1 $.

In \cite{TakqH}, the variants of the $q$-Heun equation, i.e. $( A^{\langle 3 \rangle} -E) g(x)=0 $ and $( A^{\langle 2 \rangle} -E) g(x)=0$, were discussed.
It was also shown that the two equations has quasi-exact solvability, and they have polynomial-type solutions for some special cases.
Then it would be possible to discuss the spectral polynomial, real root property of them, and analysis of the roots of the spectral polynomial by the ultradiscrete limit.

As a different direction to the variants of the $q$-Heun equation, we may consider degenerations of the $q$-Heun equation as Heun's differential equation admits degenerations such as singly confluent Heun equation, doubly confluent Heun equation, bi-confluent Heun equation and tri-confluent Heun equation by confluence of the singularities (see \cite{Ron}).
Then it would be possible to consider polynomial-type solutions to degenerations of the $q$-Heun equation.

\section*{Acknowledgements}
The authors are grateful to Simon Ruijsenaars for valuable comments and fruitful discussions.
The third author was supported by JSPS KAKENHI Grant Number JP26400122.

\appendix
\section{Ultradiscrete limit of the algebraic equation}

We consider the ultradiscrete limit (i.e. $q \to +0$) of the algebraic equation 
\begin{equation}
\sum _{j=0}^M \tilde{c}_j(q) x^j = 0.
\end{equation}
We assume that $\tilde{c}_j(q) /((-1)^{p_j} c_j q^{\lambda _j}) \to 1$ as $q \to +0$ $(j=0,1,\dots ,M)$, where $p_j \in \{0,1 \}$, $c_j >0$ and $\lambda _j \in \Rea $.

Let $P$ and $N$ be the subsets of $\{0,1,\dots ,M \}$ such that $P=\{ j \: | \: p_j=0 \}$ and $N=\{ j \: | \: p_j=1 \} $.
Then the algebraic equation is written as 
\begin{equation}
\sum _{j \in P} c_j(q) x^j = \sum _{j \in N} c_j(q) x^j ,
\label{eq:jPNeq}
\end{equation}
where $c_j(q) = (-1)^{p_j} \tilde{c}_j(q) $.
Set $x=q^t $.
Then $c_j (q) x^j /( c_j q^{\lambda _j +jt}) \to 1$ as $q \to +0$.
Since $q^{\alpha } +q^{\beta } $ can be approximated as $ q^{\min (\alpha ,\beta )}$ by the limit $q\to +0$, we obtain the following ultradiscrete equation;
\begin{equation}
\min \{ \lambda _j + jt \} _{ j\in P }= \min \{ \lambda _j + jt \} _{ j\in N } . 
\label{eq:minPN}
\end{equation}
On the solution of the ultradiscrete equation, we immediately obtain the following proposition.
\begin{prop}
Let $t_0$ be a solution to Eq.(\ref{eq:minPN}) such that the minimum is attained at $k \in P$ and $k' \in N$.
Then we have $t_0 = ( \lambda _{k '} -\lambda _{k })/(k-k')$ and $\lambda _j + j t_0 \geq  \lambda _k + k t_0 = \lambda _{k '} + {k'} t_0 $ for all $j$.
\end{prop}
Here we assume that 
\begin{equation}
\lambda _j + j t_0 >  \lambda _k + k t_0 = \lambda _{k '} + {k'} t_0 , \; j \in \{ 0,1,\dots ,M \} \setminus \{ k,k' \}, \; k\in P, \; k'\in N .
\label{eq:ineqla}
\end{equation}  
By substituting $x=c q^{t_0}$ into the algebraic equation and observing the coeffient of the leading term $q^{\lambda _k} q^{k t_0 } (= q^{\lambda _{k'}} q^{k' t_0 })$, the constant $c$ should satisfy $c= (c_k /c_{k'}) ^{1/(k'-k)}$.

We are going to justify that the value $x =(c_k /c_{k'}) ^{1/(k'-k)} q ^{t_0 } $ is asymptotic to a solution of the algebraic equation as $q \to +0$.
\begin{thm} \label{thm:implthm}
Let $P$ and $N$ be sets such that $P \cap N = \phi $ and $P \cup N= \{0,1,\dots ,M \} $, $c_j(q) $ $(j \in \{0,1,\dots ,M \})$ be a function such that  $c_j (q) / q^{\lambda _j } \to c_j $ as $q \to +0$ for some $c_j \in \Rea _{>0}$.
We assume Eq.(\ref{eq:ineqla}).
Then the algebraic equation 
\begin{equation}
\sum _{j \in P} c_j(q) x^j = \sum _{j \in N} c_j(q) x^j .
\label{eq:jPNeq0}
\end{equation}
has a solution $x= x(q) $ $(0< q <q_0 )$ such that 
\begin{equation}
\lim _{q \to +0} \frac{x(q)}{(c_k /c_{k'}) ^{1/(k'-k)} q ^{( \lambda _{k '} -\lambda _{k })/(k-k')}} =1
\end{equation}
for some $q_0 \in \Rea _{>0}$.
\end{thm}
\begin{proof}
Set $t_0 = ( \lambda _{k '} -\lambda _{k })/(k-k')$, $x = q ^{t_0 } u $ and
\begin{align}
f(u,q)= q^{-(\lambda _k +k t_0 )} \Big[ \sum _{j \in P} c_j(q) [ q ^{t_0 } u ]^j - \sum _{j \in N} c_j(q) [ q ^{t_0 } u ]^j \Big] .
\end{align}
Then the equation $f(q ^{-t_0 } x,q)=0 $ is equivalent to Eq.(\ref{eq:jPNeq0}) for $q>0$.
As $q\to +0$, we have
\begin{align}
& f(u,q)  \sim q^{-(\lambda _k +k t_0 )} \Big[ \sum _{j \in P}  c_j  q^{\lambda _j +j t_0 } u^j  - \sum _{j \in N} c_j q^{\lambda _j +j t_0 } u^j \Big] .
\end{align}
By the assumption of Eq.(\ref{eq:ineqla}), we obtain
\begin{align}
&  \lim _{q \to +0 } f(u,q) = c_k u^k  - c_{k'} u^{k'} ,
\end{align}
and define $f(u,0)$ by the limit $q\to +0$.
Then we have $f( (c_k /c_{k'}) ^{1/(k'-k)} ,0) =0$.
On the other hand, we have
\begin{align}
 \lim _{q \to +0 } \frac{\partial }{\partial u} f(u,q)  = & \lim _{q \to +0 }   q^{-(\lambda _k +k t_0 )} \Big[ \sum _{j \in P}  c_j  q^{\lambda _j +j t_0 } u^{j-1}  - \sum _{j \in N} c_j q^{\lambda _j +j t_0 } u^{j-1} \Big] \\
 = & k c_k u^{k-1} -k'c_{k'} u^{k'-1} = \frac{\partial }{\partial u} f(u,0), \nonumber \\
  \frac{\partial }{\partial u} f(u,0) |_{u=(c_k /c_{k'}) ^{1/(k'-k)} } = &  (k -k') c_k ^{(k'-1 )/(k'-k)}  c_{k'} ^{(k-1 )/(k-k')} \neq 0. \nonumber
\end{align}
By the implicit function theorem, there exists $q_0 >0$ and a continuous function $u(q)$ on $0\leq q <q_0 $ such that $f(u(q),q)=0$ and $\lim _{q\to +0} u(q) =(c_k /c_{k'}) ^{1/(k'-k)}$.
Hence Eq.(\ref{eq:jPNeq0}) has a solution $x= x(q) $ $(0< q <q_0 )$ such that 
\begin{equation}
\lim _{q \to +0} \frac{x(q)}{(c_k /c_{k'}) ^{1/(k'-k)} q ^{( \lambda _{k '} -\lambda _{k })/(k-k')}} =1.
\end{equation}
\end{proof}
\begin{thm} \label{thm:zerosasymp}
Let
\begin{equation}
\sum _{j=0}^M \tilde{c}_j(q) x^j = 0
\end{equation}
be the algebraic equation such that $\tilde{c}_{M-j}(q) /\tilde{c}_{M-j+1} (q) \to r_j q^{\mu _j}  $ as $q \to +0$ $(j=1,2,\dots ,M)$, where $r_j  \in \Rea _{\neq 0}$, $\mu _j \in \Rea $ and $\mu _1 < \mu _2 < \dots <\mu _M$.
Then there exists $q_k \in \Rea _{>0}$ for $k=1,2,\dots ,M$ such that the algebraic equation has a solution $x= x_k(q) $ for $0<q<q_k$ which is continuous on $q$ and satisfies 
\begin{equation}
\lim _{q \to +0} \frac{x _k (q)}{-r_k q^{\mu _k}} =1 .
\end{equation}
\end{thm}
\begin{proof}
Let $P$ (resp. $N$) be the subsets of $\{0,1,\dots ,M \}$ such that $P= \{ 0 \} \cup \{ j \: | \: r_1 \cdots r_j >0 \}$ and $N=\{ j \: |\:  r_1 \cdots r_j < 0 \} $.
Without loss of generality, we may assume that the algebraic equation is monic, i.e. $\tilde{c}_M(q) =1$.
Then we have $\tilde{c}_{M-j}(q) \sim r_1 \cdots r_{j}q^{\mu _1+\cdots +\mu _{j}} $ and the algebraic equation is written as 
\begin{equation}
\sum _{j \in P} |\tilde{c}_j(q)| x^j = \sum _{j \in N} |\tilde{c}_j(q)| x^j .
\label{eq:jPNeq1}
\end{equation}
We investigate positive solutions to Eq.(\ref{eq:jPNeq1}).
By setting $x=q^t $, the corresponding ultradiscrete equation is written as
\begin{equation}
\min \{ \mu _1 +\cdots + \mu _j +(M-j) t \} _{ j\in P }= \min \{ \mu _1 +\cdots + \mu _j +(M-j) t \} _{ j\in N } . 
\label{eq:minPNthmdistroots1}
\end{equation}
Let $k \in \{1,2,\dots ,M \}$ such that $r_{k} <0$.
Then we have ($k-1 \in  P$ and $k \in N$) or ($k-1 \in N$ and $k \in P$).
It follows from $\mu_1 < \mu _2 < \dots <\mu _M $ that $\mu _1 +\cdots + \mu _{k-1} +(M-k+1) \mu _{k } <  \mu _1 +\cdots + \mu _{j} +(M-j) \mu _{k }$ for $j \neq k-1, k$ and the value $t_k = \mu _{k } $ satisfies Eq.(\ref{eq:minPNthmdistroots1}).
Hence we can apply Theorem \ref{thm:implthm} and there exists $q_k \in \Rea _{>0}$ such that the algebraic equation has a solution $x= x_k(q) $ for $0<q<q_k$ which is continuous on $q$ and satisfies $x _k (q)/( |r_k | q^{\mu _k}) \to 1$ as $q \to +0$. 

We investigate negative solutions to Eq.(\ref{eq:jPNeq1}).
Let $P'$ (resp. $N'$) be the subsets of $\{0,1,\dots ,M \}$ such that $P'= \{ 0 \} \cup \{ j \: | \: r_1 \cdots r_j (-1)^{j} <0 \}$ and $N'=\{ j \: |\:  r_1 \cdots r_j (-1)^{j} > 0 \} $.
By setting $x=-q^t $, the corresponding ultradiscrete equation is written as
\begin{equation}
\min \{ \mu _1 +\cdots + \mu _j +(M-j) t \} _{ j\in P' }= \min \{ \mu _1 +\cdots + \mu _j +(M-j) t \} _{ j\in N' } . 
\label{eq:minP'N'thmdistroots1}
\end{equation}
Let $k \in \{1,2,\dots ,M \}$ such that $r_{k} >0$.
Then we have ($k-1 \in  P'$ and $k \in N'$) or ($k-1 \in N'$ and $k \in P'$).
It follows from $\mu_1 < \mu _2 < \cdots <\mu _M $ that $\mu _1 +\cdots + \mu _{k-1} +(M-k+1) \mu _{k } <  \mu _1 +\cdots + \mu _{j} +(M-j) \mu _{k }$ for $j \neq k-1, k$ and the value $t_k = \mu _{k } $ satisfies Eq.(\ref{eq:minPNthmdistroots1}).
Hence we can apply Theorem \ref{thm:implthm} and there exists $q_k \in \Rea _{>0}$ such that the algebraic equation has a solution $x= x_k(q) $ for $0<q<q_k$ which is continuous on $q$ and satisfies $- x _k (q)/(r_k q^{\mu _k}) \to 1$ as $q \to +0$. 
\end{proof}
\begin{cor} \label{cor:sol}
Assume that 
\begin{equation}
\sum _{j=0}^M \tilde{c}_j(q) x^j \sim c q^{\mu '} \prod _{j=1}^M (x + r_j q^{\mu _j}  )
\end{equation}
where $r_j  \in \Rea _{\neq 0}$, $\mu _j \in \Rea $ and $\mu _1 < \mu _2 < \cdots <\mu _M$.
Then there exists $q_k \in \Rea _{>0}$ for $k=1,2,\dots ,M$ such that the algebraic equation $\sum _{j=0}^M \tilde{c}_j(q) x^j =0 $ has a solution $x= x_k(q) $ for $0<q<q_k$ which is continuous on $q$ and satisfies 
\begin{equation}
\lim _{q \to +0} \frac{x _k (q)}{-r_k q^{\mu _k}} =1 .
\end{equation}
\end{cor}
\begin{proof}
It follows from $\mu _1 < \mu _2 < \dots <\mu _M $ that 
\begin{equation}
\prod _{j=1}^M (x + r_j q^{\mu _j}  ) \sim \sum _{j=0}^M r_1 \cdots r_j q^{\mu _1+ \cdots + \mu _j} x^{M-j}.
\end{equation}
Therefore the corollary follows from the theorem.
\end{proof}

\end{document}